\documentclass[12pt, reqno]{amsart} 
\usepackage{amssymb,amsthm,amsmath,mathtools}
\usepackage{layout}
\usepackage{enumerate}
\usepackage{ifthen}
\usepackage{graphicx} 
\usepackage[autostyle, english = american]{csquotes}
\usepackage{amsfonts}
\usepackage{amscd}
\usepackage{amssymb}
\allowdisplaybreaks
\usepackage{mathabx}
\usepackage{color}
\usepackage{amsbsy}
\usepackage{amsthm}
\usepackage{amsmath}
\usepackage{amsxtra}
\usepackage{mathrsfs}
\usepackage{bbm}
\usepackage{dsfont}
\newtheoremstyle{case}{}{}{}{}{}{:}{ }{}
\newtheorem{case}{Case}
\newtheorem{subcase}{Subcase}
\numberwithin{subcase}{case}
\usepackage{ifthen}

\usepackage[colorlinks,citecolor=red,hypertexnames=false]{hyperref} 

\usepackage[a4paper, margin=1.4in]{geometry}
\newtheorem{theorem}{Theorem}[section]
\newtheorem{cor}{Corollary}[theorem]
\newtheorem{lem}[theorem]{Lemma}

\newtheorem{prop}[theorem]{Proposition}

\theoremstyle{definition}
\newtheorem{defn}[theorem]{Definitions}
\newtheorem{remark}{Remark}
\newtheorem{conjecture}[theorem]{Conjecture}
\numberwithin{equation}{section}
\definecolor{red}{rgb}{1.0, 0.0, 0.0}

\title[]{The Exact Enumeration of $4$-nomial and $5$-nomial Multiples of the Product of Primitive Polynomials over GF(2)}
\author{Soniya Takshak}
\address{Soniya Takshak \endgraf
	Department of Mathematics \endgraf
	Indian Institute of Technology Delhi \endgraf
	New Delhi, 110016, India} 
\email{sntakshak9557@gmail.com}
\author{Rajendra Kumar Sharma}
\address{Rajendra Kumar Sharma \endgraf
	Department of Mathematics \endgraf
	Indian Institute of Technology Delhi \endgraf
	New Delhi, 110016, India} 
\email{rksharmaiitd@gmail.com}

\date{}
\begin{document}
	\allowdisplaybreaks
	
	\begin{abstract} 
		
		Linear feedback shift registers (LFSRs) are used to generate secret keys in stream cipher cryptosystems. There are different kinds of key-stream generators like filter generators, combination generators, clock-controlled generators, etc. For a combination generator, the connection polynomial is the product of the connection polynomials of constituent LFSRs. For better cryptographic properties, the connection polynomials of the constituent LFSRs should be primitive with co-prime degrees. The cryptographic systems using LFSRs as their components are vulnerable to correlation attacks. The attack heavily depends on the $t$-nomial multiples of the connection polynomial for small values of $t$. In 2005, Maitra, Gupta, and Venkateswarlu provided a lower bound for the number of $t$-nomial multiples of the product of primitive polynomials over GF(2). The lower bound is exact when $t=3$. In this article, we provide the exact number of $4$-nomial and $5$-nomial multiples of the product of primitive polynomials. This helps us to choose a more suitable connection polynomial to resist the correlation attacks. Next, we disprove a conjecture by Maitra, Gupta, and Venkateswarlu. 
			
	\end{abstract}
	\maketitle


\section{Introduction}
Linear feedback shift registers (LFSRs) are major parts in the stream cipher cryptosystems because they generate pseudorandom sequences.	
The connection polynomial associated with the LFSR needed to be primitive for cryptographic security. Generally, the connection polynomials are primitive polynomials over GF(2). Therefore, we consider polynomials over GF(2) only. There are different kinds of key-stream generators like filter generators, combination generators, clock-controlled generators, etc.\cite{types-of-stream-ciphers}. Different cryptanalytic attacks like algebraic attacks \cite{Algebraic-attacks-on-combiners-with-memory}, correlation attacks \cite{Fast-correlation-attacks-on-certain-stream-ciphers,Fast-correlation-attack-on-stream-ciphers-turbo-codes}, binary decision diagram BDD-based attacks \cite{BDD-based-Cryptanalysis-of-Stream-Cipher-A-Practical-Approach}, etc., have been developed for recovering the secret keys of LFSR-based stream ciphers.\par
It is well known that the recurrence relation of the LFSR, which is generated by the connection polynomial $f(x)$, is also satisfied by multiples of $f(x)$. If we assume that a multiple of $f(x)$ is sparse and of moderate degree, then one can apply the fast correlation attack by using the recurrence relation induced by the multiples of $f(x)$ to recover the secret key as described in \cite{Fast-correlation-attacks-on-certain-stream-ciphers}.\par
In case of combination generators, the connection polynomial is the product of the connection polynomials of constituent LFSRs. Generally, the connection polynomials of constituent LFSRs are of relatively prime degrees. To resist the correlation attack, primitive polynomials should be chosen such that the product polynomial does not have sparse multiples of lower degrees. In this article, we obtain the exact number of $4$-nomial and $5$-nomial multiples of the product of primitive polynomials with mutually co-prime degrees over GF(2). This exact count gives a better approximation for the degree of the least degree $4$-nomial and $5$-nomial multiples of the product of primitive polynomials.
\par
This article is arranged in the following manner. Section 2 consists of some fundamental definitions and results. Sections 3 and 4 contain the main results of the article. In section 3, we first obtain the exact number of $4$-nomial and $5$-nomial multiples of the product of two primitive polynomials with co-prime degrees. In section 4, we apply a recursive approach and obtain the exact number of $4$-nomial and $5$-nomial multiples of the product of $k$ primitive polynomials, where $k$ is a positive integer. Further, we discuss the degree distribution of $t$-nomial multiples, particularly for $t=4$ and $5$. In section 5, we disprove the conjecture proposed in \cite{Results-on-multiples-of-primitive-polynomials-and-their-products} regarding the least degree $t$-nomial multiples of the product of primitive polynomials over GF(2).

\section{Preliminaries}
The following notions are useful in the subsequent part of the article. Let $\mathbb{F}_{q}$ denotes a finite field with $q$ elements, where $q=p^m$, $p$ be a prime and $m$ be a positive integer. Next, $\mathbb{F}_{q}^*$ denotes the cyclic group of non-zero elements of $\mathbb{F}_{q}$. A generator of $\mathbb{F}_{q}^*$ is called a primitive element of $\mathbb{F}_{q}$. We write $f_1f_2$ to denote the product of two polynomials $f_1$ and $f_2$. Every polynomial throughout the article is over GF(2) unless otherwise stated. Let us recall some definitions and results from \cite{FFTA} and \cite{Results-on-multiples-of-primitive-polynomials-and-their-products}.

\begin{defn}\emph{\cite{FFTA}}
	An irreducible polynomial of degree $d$ is called a primitive polynomial over $\mathbb{F}_q$ if its roots are primitive elements in $\mathbb{F}_{q^d}$.
\end{defn}

\begin{defn}\emph{\cite{Results-on-multiples-of-primitive-polynomials-and-their-products}}
	A polynomial with $t$ non-zero terms, one of them being the constant term, is called a $t$-nomial, or in other words, a $t$-nomial is a polynomial of weight $t$ with a non-zero constant term.
\end{defn}

\begin{defn}\emph{\cite{FFTA}}
	Let $f \in \mathbb{F}_q[x]$ be a non-zero polynomial. If $f(0)\ne0$, then the least positive integer $e$ for which $f(x)$ divides $x^e-1$ is called the order of $f(x)$ and is denoted by $ord(f)$. The order of $f(x)$ is also called the exponent of $f(x)$.
\end{defn}

\begin{theorem}\emph{\cite{FFTA}}\label{base2}
	Let $g_{1},\dots ,g_{k}$ be pairwise relatively prime non-zero polynomials over $\mathbb{F}_{q}$, and $f = g_{1}\cdots g_{k}$. Then $ord(f)$ is equal to the least
	common multiple of $ord(g_{1}),\dots ,ord(g_{k})$.
\end{theorem}

\begin{theorem}\label{base}\emph{\cite{Results-on-multiples-of-primitive-polynomials-and-their-products}}
	Consider $k$ polynomials $f_{1}(x), f_{2}(x), \dots , f_{k}(x)$ over \textnormal{GF(2)}  having degrees $d_{1}, d_{2}, \dots , d_{k}$ and exponents $e_{1}, e_{2}, \dots , e_{k}$ respectively, with the following conditions:\\
	\noindent
	$(1)$ $e_{1}, e_{2}, \dots , e_{k}$ are pairwise co-prime,\\
	$(2)$ $f_{1}(0) = f_{2}(0) = \cdots = f_{k}(0) = 1$,\\
	$(3)$ $\text{gcd} (f_{r} (x), f_{s}(x)) = 1 \ for\  1\leq r \neq s \leq k$,\\
	$(4)$ number of $t$-nomial multiples $(\text{with degree} < e_{r})$ of $f_{r} (x)$ is $ N_{r,t}$.\\
	Then the number of $t$-nomial multiples with degree $<e_{1} e_{2} \cdots  e_{k}$ of the product polynomial
	$f_{1}f_{2} \cdots  f_{k}$ is at least $((t-1)!)^{k-1}\prod_{r = 1}^{r = k}N_{r,t}$.
\end{theorem}

\begin{cor}\emph{\cite{Results-on-multiples-of-primitive-polynomials-and-their-products}}
	Let $f_{1}(x), f_{2}(x), \dots, f_{k}(x)$ be $k$ primitive polynomials of degrees $d_{1}, d_{2}, \dots, d_{k}$ respectively, where $d_{1}, d_{2}, \dots, d_{k}$ are pairwise co-prime. Then the number of t-nomial multiples with degree $<(2^{d_{1}} - 1) (2^{d_{2}} - 1) \cdots (2^{d_{k}} - 1)$ of $f_{1}f_{2}\cdots f_{k}$ is at least $((t-1)!)^{k-1}\prod_{r = 1}^{r = k}N_{r,t}$, where $N_{r,t}$ is the number of t-nomial multiples with degree $< 2^{d_{r}} - 1$ of $f_{r}(x)$.
\end{cor}

\begin{cor}\label{cor3}\emph{\cite{Results-on-multiples-of-primitive-polynomials-and-their-products}}
	The number of $3$-nomial multiples with degree $< e_{1} e_{2} \cdots  e_{k}$ of the product $f_{1} f_{2} \cdots  f_{k}$ is $(2!)^{k-1}\prod_{r = 1}^{r = k}N_{r,3}$.
\end{cor}

Let $f(x)$ be a polynomial of exponent $e$, and $h_1, h_2, \dots, h_s $ are all the $t$-nomial multiples of $f(x)$ having degrees less than $e$. Lemma \ref{corollary} provides us the sum of the degrees of $h_1, h_2, \dots, h_s.$ We will use Lemma \ref{corollary} to find the cardinality of the set $\mathcal{A}_{r}$ in section 3. 

\begin{lem}\label{corollary}
	Let $f(x)$ be a polynomial over GF(2) of degree $d$ and exponent $e$. If the number of $t$-nomial multiples of $f(x)$ with degree $<e$ is $N_{t}$, then
	$$\sum_{i=1}^{N_{t}}d_{i} = \frac{(t-1)}{t} e N_{t} $$
	where $d_{i}$ is the degree of $t$-nomial multiple of $f(x)$.
\end{lem}
The proof follows directly from the proof of \cite[pp. 318, Corollary 2] {Results-on-multiples-of-primitive-polynomials-and-their-products}.

\section{Enumeration of $4$-nomial and $5$-nomial multiples of the product of two primitive polynomials}
Let us consider a key-stream generator in which a non-linear Boolean function $f(x_1,\dots,x_n)$ is used to combine the outputs of $n$ LFSRs. Let the Boolean function $f(x_1,\dots,x_n)$ is of $(k-1)$-th order correlation immune. In this scenario, the correlation attack can be employed by taking $k$ LFSRs together having connection polynomials $p_i(x), 1 \leq i \leq k$. The correlation attack mainly exploits the statistical dependence between the keystream and the stream generated by the block of $k$ LFSRs. To execute the attack we need $t$-nomial multiples of the product polynomial $\prod_{i=1}^{k} p_i (x)$. The implementation of the attack uses multiples of the product polynomial of low weight.
For cryptographic applications, the degrees of the primitive polynomials are taken to be co-prime to each other. Hence it is important to investigate the exact number of $t$-nomial multiples of the product of primitive polynomials, specifically for small values of $t$.\par
Theorem \ref{base} provides us with a lower bound for $t$-nomial multiples of the product of $k$ polynomials with specified conditions. The lower bound is not the exact count for $t \geq 4$. The following Proposition provides us with the exact number of $4$-nomial multiples of the product of two polynomials. The proof follows directly from \cite[Proposition 1] {Results-on-multiples-of-primitive-polynomials-and-their-products}. Later in section \ref{recursive_approach}, we apply a recursive approach to Proposition \ref{prop} and obtain the exact number of $4$-nomial multiples of the product of $k$ polynomials.

\begin{prop}\label{prop}
	Consider two polynomials $f_{1}(x), f_{2}(x)$ over \textnormal{GF(2)} of degrees $d_{1}, d_{2}$ and exponents $e_{1}, e_{2}$ respectively, such that:\\
	\noindent
	$(1)$ $gcd(e_{1}, e_{2})=1$,\\
	$(2)$ $f_{1}(0) = f_{2}(0) = 1$,\\
	$(3)$ $gcd (f_{1} (x), f_{2}(x)) = 1$,\\
	$(4)$ number of 4-nomial multiples $($with degree $ < e_{r} )$ of $f_{r} (x)$ is $ N_{r,4}$.\\
	Then the exact number of 4-nomial multiples with degree $< e_{1} e_{2}$ of the product polynomial
	$f_{1}f_{2}$ is
	\begin{center}
		$6N_{1,4}N_{2,4} + (e_{1} - 1)(e_{2} - 1) + (3(e_{1} - 1) + 1)N_{2,4} + (3(e_{2} - 1) + 1)N_{1,4}$.
	\end{center}
\end{prop}
In Theorem \ref{basemy}, we obtain the exact number of $5$-nomial multiples of the product of two primitive polynomials. First, we introduce some notations as follows.\par
Let $f_1(x)$ and $f_2(x)$ be two polynomials over GF(2), and we want to find the exact number of $5$-nomial multiples of the product polynomial $f_1f_2$. We define a set $\mathcal{A}_{r}$ corresponding to $f_r(x)$, where $r=1,2$. Let the number of $3$-nomial multiples of $f_r(x)$ be $N_{r,3}$, and $h_j(x)$ denotes a $3$-nomial multiple of $f_{r}(x)$, of degree $d_j< e_{r}$, $1 \leq j \leq N_{r,3}$.
\begin{center}
	$\mathcal{A}_{r} := \{ x^{\lambda_{j}}h_j(x)
	: \ 1 \leq j \leq N_{r,3},\  1\leq \lambda_{j} \leq e_{r}-1-d_j\}$
\end{center}
Let $n_r$ be the cardinality of $\mathcal{A}_{r}$, by Lemma \ref{corollary} we get
\begin{equation*}
	\begin{aligned}
		|\mathcal{A}_{r}| 
		= & {}\sum_{j=1}^{N_{r,3}}(e_r-1-d_j)\\
		= & {} (e_r-1)N_{r,3}-\sum_{j=1}^{N_{r,3}}d_j\\
		= & {} (e_r-1)N_{r,3}- \frac{2}{3} e_r N_{r,3}\\
		= & {} (\frac{e_r}{3}-1)N_{r,3}.
	\end{aligned}
\end{equation*}

\begin{theorem}\label{basemy}
	Consider two polynomials $f_1(x), f_2(x)$ over GF(2) of degrees $d_1,d_2$ and exponents $e_1,e_2$ respectively, such that:\\
	\noindent
	$(1)$ $gcd(e_{1}, e_{2})=1$,\\
	$(2)$ $f_{1}(0) = f_{2}(0) = 1$,\\
	$(3)$ $gcd (f_{1} (x), f_{2}(x)) = 1$\\
	$(4)$ number of $5$-nomial and $3$-nomial multiples $(\text{with degree} < e_{r})$ of $f_{r} (x)$ are $ N_{r,5}$ and $ N_{r,3}$, respectively.\\
	\noindent
	Then the exact number of $5$-nomial multiples of the product polynomial $f_1f_2$ having degree $< e_1e_2$ is equal to
	
	\parbox[t]{4.9in}{
		$4!N_{1,5} N_{2,5} + N_{1,3}N_{2,5}[12(e_1-2)+8] +N_{2,3}N_{1,5}[12(e_2-2)+8] +4![n_1N_{2,5}+n_2N_{1,5}] + 18n_1n_2 + n_{1}N_{2,3}[12(e_2-3)+14]+ n_{2}N_{1,3}[12(e_1-3)+14]+ N_{1,3}N_{2,3}[5(e_1-3)(e_2-3)+7(e_1-3)+7(e_2-3)+5]$,}
	
	\noindent
	where $n_r= |\mathcal{A}_{r}|$.
\end{theorem}
\begin{proof}
	Let $p_r(x)=x^{i_{1r}}+x^{i_{2r}}+ x^{i_{3r}} +x^{i_{4r}}+1$ be a multiple of $f_r(x)$ with degree $<e_r$, $r=1,2$. Clearly, $p_r(x)$ has some choices; it can be a $5$-nomial multiple or a $3$-nomial with an extra term $x^k+x^k$, etc. Since we are working with polynomials over GF$(2)$, these types of terms do not change the meaning of the polynomial. Now we discuss the method to find $5$-nomial multiples of the product polynomial $f_1f_2$ with the help of $p_1(x)$ and $p_2(x)$. Let us consider the system of congruences 
	\begin{center}
		$I_{s}\equiv i_{s1}\  \text{mod}\  e_1$\\
		$I_{s}\equiv i_{s2}\  \text{mod}\  e_2$
	\end{center}
	where $s= 1,2,\dots, 4$. Solving the system by the Chinese remainder theorem, we get a unique solution $\text{mod}~ e_1e_2$ for $s= 1,2,\dots, 4$.\par
	Now we will show that $x^{I_{1}}+x^{I_{2}}+x^{I_{3}} +x^{I_{4}}+1$ is a multiple of the product polynomial $f_1f_2$. For this, it is sufficient to show that $x^{I_{1}}+x^{I_{2}}+x^{I_{3}} +x^{I_{4}}+1$ is divisible by $f_r(x),\ r=1,2$. Since $e_r$ is the exponent of $f_r(x)$, we need to show that $x^{I_{1}\textrm{mod}~ e_r}+x^{I_{2}\textrm{mod}~ e_r}+x^{I_{3}\textrm{mod}~ e_r} +x^{I_{4} \textrm{mod}~ e_r}+1$ is divisible by $f_r(x)$.
	We have $i_{sr}\equiv I_{s}~\textrm{mod}~e_r, r=1,2,~ \& ~s= 1,2,\dots,4$. Thus $x^{I_{1}\textrm{mod}~ e_r} + x^{I_{2} \textrm{mod}~ e_r}+x^{I_{3}\textrm{mod}~ e_r}+x^{I_{4}\textrm{mod}~ e_r}+1$ is equal to $x^{i_{1r}}+x^{i_{2r}}+x^{i_{3r}}+x^{i_{4r}}+1$, which is $p_r(x)$. Hence $x^{I_{1}}+x^{I_{2}}+x^{I_{3}}+x^{I_{4}}+1$ is divisible by $f_r(x)$.\par
	As we require $x^{I_{1}}+x^{I_{2}}+x^{I_{3}}+x^{I_{4}}+1$ to be a $5$-nomial multiple of the product polynomial $f_1f_2$, we use the following rule. Corresponding to $p_1(x)$, we fix the powers of $x$ in order $i_{11}, i_{21}, i_{31}, i_{41}$. Now corresponding to $p_2(x)$, we consider all the possible permutations of $i_{12}, i_{22}, i_{32}, i_{42}$ such that $I_{l} \not\equiv I_{m}~\textrm{mod}~e_1e_2,$ whenever $l\ne m$.
	In this way, $x^{I_{1}}+x^{I_{2}}+x^{I_{3}} +x^{I_{4}}+1$ turns out to be a $5$-nomial multiple of the product polynomial $f_1f_2$. In the following cases, we consider all the possible choices of $p_1(x)$ and $p_2(x)$, and evaluate the corresponding number of $t$-nomial multiples of $f_1f_2$. Unless otherwise stated, we take the degrees of $p_r(x)$ and the multiples of $f_1f_2$ less than $e_r$ and $e_1e_2$, respectively.
	
	\begin{case}
		Let $p_r(x)=x^{i_{1r}}+x^{i_{2r}}+x^{i_{3r}}+x^{i_{4r}}+1$ be a $5$-nomial multiple of $f_r(x)$, $r=1,2$. As discussed above, we fix the powers of $x$ in $p_1(x)$ in order $i_{11}, i_{21}, i_{31}, i_{41}$. Now we have $4!$ number of permutations of $i_{12}, i_{22}, i_{32}, i_{42}$ such that $I_{l} \not\equiv I_{m}~\textrm{mod}~e_1e_2,$ when $l\ne m$. Hence we get $4!N_{1,5} N_{2,5}$ many distinct $5$-nomial multiples of $f_1f_2$ in this case.	
	\end{case}
	
	\begin{case}
		Let $x^{i_1}+x^{i_2} +1$ be a $3$-nomial multiple of $f_{1}(x)$, choose $p_1(x)= x^{i_1}+x^{i_2}++x^k+x^k+1, 0\leq k < e_1$. Let $p_2(x)=x^{j_{1}}+x^{j_{2}}+ x^{j_{3}} +x^{j_{4}}+1$ be a $5$-nomial multiple of $f_{2}(x)$. Now, we fix the powers of $x$ in $p_1(x)$ in the order $i_{1},i_{2},k,k$, then $j_1,j_2,j_3,j_{4}$ can be placed in $\frac{4!}{2!}$ different ways if $k\ne i_1,i_2$, and $\frac{4!}{3!}$ ways if $k = i_1$ or $i_2$. Hence we get $12(e_1-2)N_{1,3}N_{2,5}+8N_{1,3}N_{2,5}$ many distinct $5$-nomial multiples of $f_1f_2$.\\
		\noindent
		Similarly, let $p_1(x)=x^{i_{1}}+x^{i_{2}}+ x^{i_{3}} +x^{i_{4}}+1$ be a $5$-nomial, and $p_2(x)=x^{j_1} +x^{j_2} +x^l +x^l +1, 0\leq l < e_2$ be a $3$-nomial, then we get $12(e_2-2)N_{1,5}N_{2,3}+8 N_{1,5}N_{2,3}$ many distinct $5$-nomial multiples of $f_1f_2$.	
	\end{case}
	
	\begin{case}
		Let $x^{a_{1}}+x^{a_{2}}+ x^{a_{3}}$ be a polynomial in $\mathcal{A}_{1}$, then we can take $p_1(x)= x^{a_{1}}+x^{a_{2}}+ + x^{a_{3}}+1+1$. Let $p_2(x)=x^{j_{1}}+x^{j_{2}}+ x^{j_{3}} +x^{j_{4}}+1$ be a $5$-nomial multiple of $f_2(x)$. Then we get $4!n_{1}N_{2,5}$ many distinct $5$-nomial multiples of $f_1f_2$.\\
		\noindent
		Similarly, if $p_1(x)=x^{i_{1}}+x^{i_{2}}+ x^{i_{3}} +x^{i_{4}}+1$ to be a $5$-nomial, and $p_2(x)= x^{b_{1}}+x^{b_{2}}+ x^{b_{3}}+1+1$, where $x^{b_{1}}+x^{b_{2}} + x^{b_{3}} \in \mathcal{A}_{2}$. Then we get $4!n_{2}N_{1,5}$ many distinct $5$-nomial multiples of $f_1f_2$.	
	\end{case}
	
	\begin{case}
		Let $x^{a_{1}}+x^{a_{2}} + x^{a_{3}} \in \mathcal{A}_{1}$ and $x^{b_{1}}+x^{b_{2}} + x^{b_{3}} \in \mathcal{A}_{2}$. Then we can take $p_1(x)= x^{a_{1}}+x^{a_{2}} + x^{a_{3}}+1+1$ and $p_2(x)= x^{b_{1}}+x^{b_{2}} + x^{b_{3}}+1+1$. Now we fix the powers of $x$ in $p_1(x)$ in order $a_1,a_2, a_{3},0$, then the number of permutations of $b_1,b_2, b_{3},0$ such that 0 never comes at the $4$-th position, is $4!-3!$. Hence, we get $18n_1n_2$ many distinct $5$-nomial multiples of $f_1f_2$.	
	\end{case}
	
	\begin{case}
		Let $p_1(x)= x^{a_{1}}+x^{a_{2}} + x^{a_{3}}+1+1$, where $x^{a_{1}}+x^{a_{2}} + x^{a_{3}} \in \mathcal{A}_{1}$, and $p_2(x)= x^{j_1}+x^{j_2}+x^l +x^l +1$, $0 \leq l < e_2$, be a $3$-nomial multiple of $f_{2}(x)$. Then the following cases arise.
		\begin{subcase}{When $l \ne 0$.}
			Then the possible number of permutations are $\frac{4!}{2!}$ if $l\neq j_1,j_2$, and $\frac{4!}{3!}$ if $l = j_1$ or $j_2$. Hence we get
			$12(e_2-3)n_{1}N_{2,3}+ 8n_{1}N_{2,3}$ many distinct $5$-nomial multiples of $f_1f_2$.
		\end{subcase}
		\begin{subcase}{When $l=0$.}
			In this case, the powers of $x$ in $p_1(x)$ are fixed in order $a_1,b_1,c_1,0$. Then $j_1,j_2,0,0,$ can be permuted in $\frac{4!}{2!}-3!$ ways such that $0$ never comes at the $4$-th position. Hence we get $6n_1N_{2,3}$ many distinct $5$-nomial multiples of $f_1f_2$ in this case.
		\end{subcase}
		\noindent
		Similarly, when $p_1(x)= x^{i_1}+x^{i_2}+x^k+x^k+1$, $0\leq k < e_1$, and $p_2(x)= x^{b_{1}}+x^{b_{2}} + x^{b_{3}}+1+1$, where $x^{b_{1}}+x^{b_{2}} + x^{b_{3}} \in \mathcal{A}_{2}$. Then we get
		$12(e_1-3)n_{2}N_{1,3}+8n_{2}N_{1,3}+ 6n_2N_{1,3}$ many distinct $5$-nomial multiples of $f_1f_2$.
	\end{case}
	
	\begin{case}
		Let $p_1(x)=x^{i_1}+x^{i_2} + x^{k} + x^{k} + 1, 0\leq k < e_1$, and $p_2(x)= x^{j_1}+x^{j_2} + x^{l}+x^{l}+1, 0\leq l < e_2,$ are $3$-nomial multiples of $f_{1}(x)$ and $f_{2}(x)$, respectively. Now the following cases arise.
		\begin{subcase}{When $k \neq 0, l\neq 0$.}
			Then the powers of $x$ in $p_2(x)$ can be placed in $\frac{1}{2!}\{\frac{4!}{2!}-2!\}$ different ways if $k\ne i_1,i_2, l\neq j_1,j_2$, in
			$\frac{1}{2!}\{\frac{4!}{3!}-2!\}$ ways if $k=i_1$ or $i_2, l\neq j_1,j_2$ or $k\neq i_1,i_2, l=j_1$ or $j_2$.\\
			This gives  
			$5(e_1-3)(e_2-3)N_{1,3}N_{2,3}+2(e_2-3)N_{1,3}N_{2,3}+2(e_1-3)N_{1,3}N_{2,3}$ many distinct $5$-nomial multiples of $f_1f_2$.	
		\end{subcase}
		\begin{subcase}{When $k=0$ but $l\neq 0$.}
			Then the powers of $x$ in $p_2(x)$ can be arranged in $\frac{1}{2!}\{\frac{4!}{2!}-2!\}$ ways if  $l\neq j_1,j_2$, and in $\frac{1}{2!}\{\frac{4!}{3!}-2!\}$ ways if $l = j_1$ or $j_2$, to get distinct $5$-nomial multiples of $f_1f_2$. Hence we get $5(e_2-3)N_{1,3}N_{2,3} + 2N_{1,3}N_{2,3}$ many $5$-nomial multiples of $f_1f_2$.
		\end{subcase}
		\begin{subcase}{When $k\ne 0$ but $l=0$.}
			In this case, we get $5(e_1-3)N_{1,3}N_{2,3} + 2N_{1,3}N_{2,3}$ many distinct $5$-nomial multiples of $f_1f_2$.
		\end{subcase}
		\begin{subcase}{When $k=0\  \& \ l=0$.}
			In this case, we get $N_{1,3}N_{2,3}$ many distinct $5$-nomial multiples of $f_1f_2$.	
		\end{subcase}
	\end{case}

	Now we add the number of $5$-nomial multiples of all the above cases and obtain the formula given in the statement of Theorem \ref{basemy}.
	To show that there is no other $5$-nomial multiple of $f_1f_2$, let $x^{I_1}+x^{I_2}+x^{I_3}+x^{I_4}+1$ be a $5$-nomial multiple of $f_1f_2$. Since $f_1f_2$ divides $x^{I_1}+x^{I_2}+x^{I_3}+x^{I_4}+1$, this implies $f_r(x)$ divides $x^{I_1}+x^{I_2}+x^{I_3}+x^{I_4}+1$, for $r=1,2$. Hence $f_r(x)$ divides $f'_r(x)$, where $f'_r(x)=x^{I_{1}\textrm{mod}~ e_r}+x^{I_{2}\textrm{mod}~ e_r}+x^{I_{3}\textrm{mod}~ e_r}+x^{I_{4}\textrm{mod}~ e_r}+1$. Suppose $f'_r(x)$ is neither a $3$-nomial nor a $5$-nomial, and also $f'_r(x)$ doesn't belong to $\mathcal{A}_{r}$. This is possible only when $I_{1}\textrm{mod}~ e_r  \equiv I_{2}\textrm{mod}~ e_r \equiv I_{3}\textrm{mod}~ e_r  \equiv 0, I_{4}\textrm{mod}~ e_r \not \equiv 0$ or $I_{1}\textrm{mod}~ e_r \equiv I_{2}\textrm{mod}~ e_r \equiv I_{3}\textrm{mod}~ e_r \equiv I_{4}\textrm{mod}~ e_r$. This is a contradiction to the fact that $f_r(x) $ divides $f'_r(x)$. Hence, there is no other $5$-nomial multiple of $f_1f_2$ than the multiples enumerated in Theorem \ref{basemy}.
\end{proof}
\noindent
In particular, Theorem \ref{basemy} is applicable when $f_1(x)$ and $f_2(x)$ are primitive polynomials over GF(2) with mutually co-prime degrees $d_1$ and $d_2$.

\section{Exact enumeration of $t$-nomial multiples of the product of $k$ primitive polynomials}\label{recursive_approach}
Let $f_1(x),f_2(x), \dots, f_k(x)$ be $k$ polynomials over GF(2) of degrees $d_{1},d_{2}, \dots ,d_{k}$ and exponents $e_{1},e_{2},\dots,e_{k}$ respectively, with the following conditions:
\begin{enumerate}
	\item  $e_{1}, e_{2}, \dots , e_{k}$ are pairwise relatively prime,
	\item $f_{1}(0) = f_{2}(0) = \cdots = f_{k}(0) = 1$,
	\item $\text{gcd} (f_{l}(x), f_{m}(x)) = 1 \ \text{for}\  1\leq l \neq m \leq k$ ,
\end{enumerate}
This section describes a recursive approach to find the exact number of $t$-nomial multiples having degree $< e_{1}e_{2} \cdots e_{k}$ of the product polynomial $f_1f_2\cdots f_k$. We denote $f_1f_2\cdots f_r$ by $f_{12\cdots r}$, and the total number of $t$-nomial multiples of $f_{12\cdots r}$ by $N_{12\cdots r,t}~ (1\leq r \leq k)$. So we can write $f_1f_2\cdots f_k=f_{12\cdots r} f_{r+1}\cdots f_k$. By Theorem \ref{base2}, the exponent of $f_{12\cdots r}$ is equal to $e_{1}e_{2} \cdots e_{r}$. Now we can see that
\begin{enumerate}
	\item  $\text{gcd}(e_{1}e_{2} \cdots e_{r}, e_{r+1})=1,\  1\leq r < k $,
	\item $f_{12\cdots r}(0)= 1,\  1\leq r < k $,
	\item $\text{gcd} (f_{12\cdots r}, f_{r+1}) = 1, \  1\leq r < k$ ,
\end{enumerate}
\noindent
Since $f_{12\cdots r}$ and $f_{r+1}$ satisfy all the conditions of Proposition \ref{prop} and Theorem \ref{basemy}, we can enumerate $t$-nomial multiples of the product of $r+1$ polynomials, $1\leq r < k$, if we have the formula for exact enumeration of $t$-nomial multiples of the product of two polynomials. As we have the formula for $t=4$ and $5$, we can apply this recursive approach to these values of $t$. Let us see how the recursive method exactly works.\par
First, we determine the number of $t$-nomial multiples of $f_{12}$. After that, we find the number of $t$-nomial multiples of $f_{123}$ by considering $f_{123}$ as the product of $f_{12}$ and $f_3$. Repeating this procedure for $ 3\leq r < k$, we can find the total number of $t$-nomial multiples of $f_1f_2\cdots f_k$ in $k-1$ steps.
\begin{remark}
	In particular, let $f_1(x), f_2(x), \dots, f_k(x)$ be $k$ primitive polynomials of pairwise co-prime degrees $d_{1}, d_{2}, \dots , d_{k}$, respectively. We can find the total number of $t$-nomial multiples of the product polynomial $f_1f_2\cdots f_k$ with degree $< (2^{d_{1}} - 1)(2^{d_{2}} - 1) \cdots (2^{d_{k}} - 1)$ using the above recursive method.
\end{remark}
\begin{remark}
	It is clear that the exact count is independent of the choice of primitive polynomials. For verification, we used SageMath\cite{sage1} to count the number of $t$-nomial multiples of the product of two and three primitive polynomials when $t=4,5$. The number turned out to be the same as we obtained using the recursive approach (refer to Table 1 and Table 2).
\end{remark}
\begin{remark}
	In this article, we have obtained the exact number of $t$-nomial multiples of the product of primitive polynomials, when $ t=4,5$. Extending these results for $t\geq 6$ gives a quite cumbersome formula. So we do not discuss these cases here.
\end{remark}

\subsection{Experimental Results}
We have verified our results with the help of SAGEMATH \cite{sage1} software for the product of two and three primitive polynomials over $\mathbb{F}_2$. For the product of two primitive polynomials, we verified our results by taking the values $(d_1, d_2): (2,3), (3,4), (3,5)$ and $(4,5)$. Moreover, we verified our results for the product of three primitive polynomials of degrees $2$, $3$ and $5$. The computations confirm that the number of $4$-nomial and $5$-nomial multiples of the product polynomial matches our results.

\subsection{An illustrative example}
Let $f_{1}(x),f_{2}(x)$ and $f_{3}(x)$ are primitive polynomials over GF(2) of degrees $d_1=2,d_2=3,$ and $d_3=5$, respectively. Hence\\
$e_1=3,e_2=7,e_3=31$;\\
$N_{1,3}=1, N_{2,3}=3, N_{3,3}=15$;\\
$N_{1,5}=0, N_{2,5}=0, N_{3,5}=840$;\\
$n_1=0, n_2= 4, n_3=140, n_{12}=36$.\\
Then the total number of 3-nomial multiples of $f_{1}f_{2}$ with degree $<e_1e_2$ is $N_{12,3}=6$, using Corollary \ref{cor3}.
Next, the total number of 5-nomial multiples of $f_{1}f_{2}$ is $N_{12,5}= 155$, by the formula obtained in Theorem \ref{basemy}.\par
Recursively, we get the total number of 5-nomial multiples of the product polynomial $f_{1}f_{2}f_{3}$ as follows.
\begin{equation*}
	\begin{aligned}
		N_{123,5}= {} & 4!N_{{12},5} N_{3,5} + N_{{12},3}N_{3,5}[12(e_{12}-2)+8] + N_{3,3}N_{{12},5}[12(e_3-2)+8]\\
		&  + 18n_{12}n_3 + 4![n_{12}N_{3,5}+n_3N_{{12},5}] + n_{12}N_{3,3}[12(e_3-3)+14]\\
		& + n_{3}N_{{12},3}[12(e_{12}-3)+14]
		+ N_{{12},3}N_{3,3}[5(e_{12}-3)(e_3-3)\\
		& +7(e_{12}-3)+7(e_3-3) + 5]
	\end{aligned}
\end{equation*}
Hence, we finally get $N_{123,5}= 7117650$.

\subsection{Degree Distribution}
In this section, we will estimate the degree of the least degree $t$-nomial multiples of the product of primitive polynomials.
To make a cryptographic system robust against correlation attack, primitive polynomials should be chosen in such a way that their product polynomial does not have sparse multiples of lower degrees.\par
Let $f_1(x), f_2(x), \dots, f_k(x)$ are $k$ primitive polynomials of degree $d_1,d_2, \dots, d_k$ and exponents $e_1,e_2, \dots, e_k$. Denote the product polynomial $f_1(x)f_2(x)\cdots f_k(x)$ by $f_{12\cdots k}(x)$. Let $d=\sum_{i=1}^{k}d_i$, and $e=(2^{d_1}-1)(2^{d_2}-1)\cdots(2^{d_k}-1)$. Also, $\mathcal{N}$ denotes the total number of $t$-nomial multiples of the product polynomial.
A simple algorithm to find the least degree $t$-nomial multiple of the product polynomial can be written as follows:
\par For $j=d$ to $e-1$.
\par 1. Consider all the $t$-nomials $g_j(x)$ of degree $j$.
\par 2. If $f_{12\cdots k}(x)$ divides $g_j(x)$ then return this $t$-nomial and terminate.\\
We associate a $(t-1)$-tuple $(i_{t-1},i_{t-2},\dots,i_1)$ to the $t$-nomial $x^{i_{t-1}}+x^{i_{t-2}}+ \cdots + x^{i_1} + 1$. The total number of $(t-1)$-tuples, where components take values from $1$ to $e-1$, is $\tbinom{e-1}{t-1}$. Let $\mathcal{N}$ denotes the number of $t$-nomial multiples of the product polynomial $f_{12\cdots k}(x)$.
Then the probability that a $(t-1)$-tuple represents a genuine $t$-nomial multiple of the product polynomial is $\mathcal{N}/\tbinom{e-1}{t-1}$. Thus, the number of $t$-nomial multiples of degree less than or equal to $c$ is $\tbinom{c}{t-1}\mathcal{N}/\tbinom{e-1}{t-1}$ approximately.\par
To estimate the deree of least degree $t$-nomial multiple, we have to find the approximate value of $c$ such that $\tbinom{c}{t-1}\mathcal{N}/\tbinom{e-1}{t-1} \approx 1$. When we estimate using the lower bound obtained in Theorem \ref{base}, as discussed in \cite[pp. 341] {Results-on-multiples-of-primitive-polynomials-and-their-products}, $c \approx 2^{d/{t-1}}$. But, in case of $4$-nomial and $5$-nomial multiples, we have the exact value of $\mathcal{N}$, which provides a better approximation for $c$. For example, in case of $4$-nomial multiples of the product of $f_1(x)=x^5 + x^2 + 1$ and $f_2(x) = x^3 + x + 1$, we get $c \approx 2^{8/3}\approx 6.3496$ using $c \approx 2^{d/{t-1}}$. But the least degree $4$-nomial multiple has degree $13$, which is the same as we get using the exact count. Similarly, in case of $5$-nomial multiples of the product of $f_1(x)=x^5 + x^4 + x^3 +x^2 + 1$ and $f_2(x) = x^7 + x + 1$, $c \approx 2^{3}$. But the least degree $5$-nomial multiple has degree $22$, which is closer to $c \approx 20$ obtained by the exact count. 
\section{A Conjecture:}
Maitra, Gupta, and Venkateswarlu \cite{Results-on-multiples-of-primitive-polynomials-and-their-products} provided Conjecture \ref{conjecture} on the least degree $t$-nomial multiples of the product of primitive polynomials. The conjecture is not true in general. In support of that, we are providing two counter-examples. First, we state the conjecture here.
\begin{conjecture}\label{conjecture}\emph{\cite{Results-on-multiples-of-primitive-polynomials-and-their-products}}
Let $x^{I_{1}}+x^{I_{2}}+\cdots+x^{I_{t-1}}+1$ be the least degree $t$-nomial multiple of the product polynomial $f_{1}f_{2}\cdots f_{k}$, which itself is a $\tau-$nomial $(4 \leq t < \tau)$. Each polynomial $f_r(x)$ is a primitive polynomial of degree $d_r$, and exponent $e_r, r = 1, \dots, k$, and degrees are pairwise coprime. Moreover, $N_{r,t} > 0$, for $ r = 1, \dots , k$. Then $I_v\not\equiv I_w\ \text{mod}\  e_r$, for any $1\leq v \neq w\leq t-1$, and for any $r = 1, \dots , k$. That is, the least degree t-nomial multiple of the product polynomial $f_{1}f_{2}\cdots f_{k}$ is the one that is generated as described in the proof of Theorem \ref{base}.
\end{conjecture}
\noindent
\textit{Counter Examples.}
Let $f_1(x)= x^4+x+1$ and $f_2(x)= x^9+x^6+x^4+x^3+1$. The product $f_{1}f_{2}=x^{13}+x^{9}+x^{8}+x^{6}+x^5+x^4+x^3+x+1$ is a $9$-nomial, and the least degree $5$-nomial multiple of $f_{1}f_{2}$ is $x^{19}+x^{17}+x^{8}+x^{4}+1$, for which $I_1 \equiv I_4\equiv 4$ mod 15.\par
Again, let $f_1(x)= x^4+x+1, f_2(x)= x^5+x^4+x^3+x^2+1$, and $f_3(x)= x^9+x^8+x^6+x^5+1$. The product $f_{1}f_{2}f_{3} = x^{18}+x^{13}+x^{11}+x^{10}+x^8+x^7+x^5+x^4+x^2+x+1$ is an $11$-nomial. Further, the least degree $4$-nomial multiple of $f_{1}f_{2}f_{3}$ is $x^{135}+x^{92}+x^{47}+1$, for which $I_2 \equiv I_3\equiv 2$ mod $15$. These two examples contradict the Conjecture \ref{conjecture}.

\section*{Acknowledgment} This work has been supported by CSIR, New Delhi, Govt. of India, under Grant F. No. 09/086(1328)/2018-EMR-1.


\begin{thebibliography}{99}
		
		
		
		
		
		
		
	{\normalsize
        \bibitem{sage1} The Sage Development Team, {\it Sage Mathematics Software}, Version 9.2, (2020).
    
        \bibitem{Algebraic-attacks-on-combiners-with-memory} Frederik Armknecht and Matthias Krause, Algebraic Attacks on Combiners with Memory, {\it Advances in Cryptology - CRYPTO 2003}, 162--175, (2003).

		\bibitem{Fast-correlation-attack-on-stream-ciphers-turbo-codes} T. Johansson and Fredrik J{\"o}nsson, Fast Correlation Attacks on Stream Ciphers based on Turbo Codes techniques, {\it EUROCRYPT, Lecture Notes in Computer Science}, 1666, 181--197, (1999).
		
		\bibitem{FFTA} Rudolf Lidl and Harald Niederreiter, Introduction to Finite Fields and their Applications, {\it Cambridge University Press}, (1994).
		
		
		\bibitem{Results-on-multiples-of-primitive-polynomials-and-their-products} Subhamoy Maitra, Kishan Chand Gupta, and Ayineedi Venkateswarlu, Results on multiples of primitive polynomials and their products over \uppercase{GF(2)}, {\it Theor. Comput. Sci.}, 341, 311--343, (2005).
		
		\bibitem{Fast-correlation-attacks-on-certain-stream-ciphers} Willi Meier, and Othmar Staffelbach, Fast correlation attacks on certain stream ciphers, {\it Journal of Cryptology}, 1, 159--176, (1989).
		
		\bibitem{types-of-stream-ciphers} Rainer A. Rueppel, Analysis and Design of Stream Ciphers, {\it Springer-Verlag}, (1986).
		
		\bibitem{BDD-based-Cryptanalysis-of-Stream-Cipher-A-Practical-Approach} H.K. Sahu,  Indivar Gupta,  N.R.  Pillai, and  R.K. Sharma, \uppercase{BDD} based Cryptanalysis of Stream Cipher: A Practical Approach, {\it IET Information Security}, 11(3), 159--167, (2017).
		}
		
		 
	
		
		
		
		
		
		
		
	\end{thebibliography}
\end{document}